\documentclass[12pt]{amsart}
\usepackage{amssymb}
\usepackage{mathtools}
\usepackage{epsfig}
\usepackage{amscd}
\usepackage{amsmath}
\usepackage{graphicx}
\usepackage{wasysym}
\usepackage{enumerate}
\usepackage{ifpdf}
\usepackage{amsfonts}
\usepackage{amsthm}
\usepackage[hyperfootnotes=false]{hyperref}
\hypersetup{hidelinks}
\usepackage{setspace}
\usepackage[usenames]{xcolor}
\usepackage{tikz}

\numberwithin{equation}{section}

\def\1#1{\overline{#1}}
\def\2#1{\widetilde{#1}}
\def\3#1{\widehat{#1}}
\def\4#1{\mathbb{#1}}
\def\5#1{\frak{#1}}
\def\6#1{{\mathcal{#1}}}

\newcommand{\R}{\mathbb R}

\newcommand{\C}{\mathbb C}

\newcommand{\B}{\mathbb B}

\newcommand{\D}{\mathbb D}

\newcommand{\N}{\mathbb N}

\theoremstyle{theorem}

\def\id{{\sf id}}

\newtheorem{theorem}{Theorem}[section]
\newtheorem{lemma}[theorem]{Lemma}
\newtheorem{proposition}[theorem]{Proposition}

\newtheorem{conjecture}[theorem]{Conjecture}

\theoremstyle{definition}
\newtheorem{definition}[theorem]{Definition}

\theoremstyle{remark}
\newtheorem{remark}[theorem]{Remark}

\numberwithin{equation}{section}

\title[convex maps]{A proof of the Muir-Suffridge conjecture for convex maps of the unit ball in $\C^n$}

\author[F. Bracci]{Filippo Bracci$^\dag$}
\address{F. Bracci: Dipartimento di Matematica, Universit\`a di Roma ``Tor Vergata", Via della Ricerca
Scientifica 1, 00133, Roma, Italia.} \email{fbracci@mat.uniroma2.it}

\author[H. Gaussier]{Herv\'e Gaussier$^\ddag$}

\address{H. Gaussier: Univ. Grenoble Alpes, CNRS, IF, F-38000 Grenoble, France}\email{herve.gaussier@univ-grenoble-alpes.fr}


\thanks{$^\dag\,$Partially supported by GNSAGA of INDAM}
\thanks{$^\ddag\,$Partially supported by ERC ALKAGE}
\thanks{This paper was written as part of the  2016-17 CAS project {\sl Several Complex Variables and Complex Dynamics}}

\long\def\REM#1{\relax}

\begin{document}
\maketitle

\begin{abstract}
We prove (and improve) the Muir-Suffridge conjecture for holomorphic convex maps. Namely, let $F:\B^n\to \C^n$ be a univalent map from the unit ball whose image $D$ is convex. Let $\mathcal S\subset \partial \B^n$ be the set of points $\xi$ such that   $\lim_{z\to \xi}\|F(z)\|=\infty$. Then we prove that $\mathcal S$ is either empty, or contains one or two points and $F$ extends as a homeomorphism $\tilde{F}:\overline{\B^n}\setminus \mathcal S\to \overline{D}$. Moreover, $\mathcal S=\emptyset$ if $D$ is bounded,  $\mathcal S$ has one point if $D$ has one connected component at $\infty$ and  $\mathcal S$ has two points if $D$ has two connected components at $\infty$ and, up to composition with an automorphism of the ball and renormalization, $F$ is an extension of the strip map in the plane to higher dimension.
\end{abstract}

\sloppy

\section{Introduction}

Let $\B^n:=\{z\in \C^n: \|z\|<1\}$ be the unit ball of $\C^n$, $n\geq 1$. A domain $D\subset \C^n$ is  {\sl convex}  if  for every $z, w\in D$ the real segment joining $z$ and $w$ is contained in $D$. 

In the paper \cite{MS1} J. Muir and T. Suffridge made the following  conjecture:
\begin{conjecture}\label{conj}
Let $F:\B^n\to \C^n$ be a univalent map such that $F(0)=0$ and $dF_0=\id$. Suppose that $D:=F(\B^n)$ is convex. Then
\begin{itemize}
\item[(a)] $D$ is bounded and $F$ extends continuously to $\partial \B^n$, or
\item[(b)] $F$ extends continuously to $\partial \B^n$ except for one point that is an infinite discontinuity, or
\item[(b)] up to pre-composing with an automorphism of $\B^n$ and post-composing with an affine transformation, there exists a holomorphic function $H:\B^{n-1}\to \C^n$, continuous up to $\partial \B^{n-1}$, such that
\begin{equation}\label{formaMS}
F(z_1,z')=\left(\frac{1}{2}\log \frac{1+z_1}{1-z_1}  \right)e_1+H\left (\frac{z'}{\sqrt{1-z_1^2}}\right).
\end{equation}
\end{itemize}
\end{conjecture}
Here, $z=(z_1,z')\in \C\times \C^{n-1}$, $e_1=(1,0,\ldots, 0)$ and a point $\xi\in \partial \B^n$ is a point of infinite discontinuity for $F$ if $\lim_{z\to \xi}\|F(z)\|=\infty$.

We point out explicitly that no regularity assumptions on the boundary of $D$ are made.

Since the boundary of a convex domain in $\C$ is a Jordan curve,  the conjecture is true for $n=1$ because of Carath\'eodory's extension theorem. In case $n>1$, in \cite[Theorem~2.13]{MS1}, Muir and Suffridge proved that if there exists $v\in \C^n$, $\|v\|=1$, such that $tv\subset D$ for all $t\in \R$, then $\xi_1:=\lim_{t\to - \infty} F^{-1}(tv)$ and $\xi_2:=\lim_{t\to  \infty} F^{-1}(tv)$ exist and, if $\xi_1\neq \xi_2$, then $F$ has the form \eqref{formaMS} (although they could not prove that $H$ extends continuously up to $\partial \B^{n-1}$). In \cite{MS2}, the same authors proved that, in case $\xi_1=\xi_2$ and there is only one infinite singularity whose span contains the direction $v$, then $F$ has a simple form for which the Conjecture holds. 

In the recent paper \cite{BG} we proved, as a consequence of a prime ends-type theory (called ``Horosphere theory'') we developed there, that the conjecture is true if $D$ is bounded and, moreover, in that case, $F$ extends as a homeomorphism up to $\overline{\B^n}$. 

In this paper we prove Conjecture \ref{conj}. In fact, we prove a stronger result:

\begin{theorem}\label{main-intro}
Let $F:\B^n \to D\subset \C^n$ be a biholomorphism. Suppose that $D$ is convex. Then 
\begin{enumerate}
\item  $D$ is bounded and $F$ extends as a homeomorphism from $\overline{\B^n}$ onto $\overline{D}$, or
\item $D$ is unbounded and has one connected component at $\infty$,  there exists a unique point $\xi\in \partial \B^n$ such that $\lim_{z\to \xi}\|F(z)\|=\infty$ and $F$ extends as a homeomorphism from $\overline{\B^n}\setminus \{\xi\}$ onto $\overline{D}$, or
\item  $D$ is unbounded and has two connected components at $\infty$,  there exist $\xi_1, \xi_2\in \partial \B^n$, $\xi_1\neq\xi_2$, such that $\lim_{z\to \xi_1}\|F(z)\|=\lim_{z\to \xi_2}\|F(z)\|=\infty$ and $F$ extends as a homeomorphism from $\overline{\B^n}\setminus \{\xi_1, \xi_2\}$ onto $\overline{D}$.
\end{enumerate}
\end{theorem}
Here, we say that an unbounded convex domain $D\subset \C^n$ has one connected component at $\infty$ if for every compact ball $B$, the  set $D\setminus B$ has one unbounded connected component, otherwise we say that it has two connected components at $\infty$ (see Section \ref{geo-convex} for precise definitions and statements). 

The Muir-Suffridge Conjecture \ref{conj} follows then  from Theorem \ref{main-intro}. Indeed, we show that a convex domain has two connected components at $\infty$ if and only if there exists $v\in \C^n$, $\|v\|=1$, such that $z+tv\subset D$ for all $z\in D$ and $t\in \R$, $v$ is the unique ``direction at $\infty$'' for $D$ and,  in such a case, $\lim_{t\to - \infty} F^{-1}(tv)\neq \lim_{t\to  \infty} F^{-1}(tv)$ (see Lemma \ref{two-components} and Proposition \ref{singular}). Therefore, up to composing with an automorphism of $\B^n$ in order to have the infinite discontinuities in $e_1$ and $-e_1$ and  post-composing with an affine transformation in such a way that $F(0)=0$ and $dF_0=\id$, $F$ has precisely the form \eqref{formaMS} by \cite[Theorem 2.13]{MS1}.

The proof of Theorem \ref{main-intro} uses many ingredients: the properties of real geodesics for the Kobayashi distance of $\B^n$, Gromov's theory of quasi-geodesics in $D$,  the theory of continuous one-parameter semigroups of $\B^n$ and that of commuting holomorphic self-maps of $\B^n$, and properties of horospheres and horospheres sequences of $D$.

We remark that, since all the ingredients we use are available for strongly convex domains with $C^3$ boundary, using Lempert's theory \cite{Lem}, one can prove that Theorem~\ref{main-intro} holds replacing $\B^n$ with any strongly convex domain with $C^3$ boundary. We leave the details of this generalization to the interested readers. 

The outline of the paper is the following. In Section \ref{geo-convex} we prove some simple geometrical properties of unbounded convex domains. In Section \ref{preli}, in order to make the paper as self-contained as possible, we collect the results and known facts we need for our proof. 

Then, in Proposition~\ref{continuous-ext} we show that $F$ extends continuously to $\partial\B^n$ with image  the closure of $D$ in the one-point compactification of $\C^n$. The proof of such a result uses Gromov's theory of quasi-geodesics and properties of the Kobayashi distance in $D$. 

Next, in Proposition~\ref{singular}, we deal with points in $\partial \B^n$ which are mapped to $\infty$ by $F$ and we prove that those points correspond to the connected components at $\infty$ of $D$ (which might be one or two). The proof uses results from the theory of continuous one-parameter semigroups in $\B^n$: indeed, if $\{p_n\}\subset D$ is a sequence converging to $\infty$ and $\frac{p_n}{\|p_n\|}\to v$, then $z+tv\in D$ for all $z\in D$ and all $t\geq 0$ (and we call such a $v$ a ``direction at $\infty$'' for $D$). Thus one can define a (continuous one-parameter) semigroup in $\B^n$ with no fixed points by setting $\phi_t(\zeta):=F^{-1}(F(\zeta)+tv)$, $\zeta\in \B^n$ and $t\geq 0$. Using horospheres and properties of horospheres defined by sequences in $D$, we show essentially that  the infinite discontinuities of $F$ are the Denjoy-Wolff points of such semigroups. We then conclude our proof by using the fact that  every two   ``directions  at $\infty$'' for $D$ give rise to two commuting  semigroups and commuting holomorphic self-maps of $\B^n$ do not share the same Denjoy-Wolff point, unless they have very particular forms. The case of two connected components at $\infty$ for $D$ is handled by using simple properties of real geodesics for the Kobayashi distance in $\B^n$. 

Finally, Theorem \ref{main-intro} is proved in Section \ref{final} (see Theorem \ref{main}).

\medskip

This paper was written while both authors were visiting the Center for Advanced Studies in Oslo for the  2016-17 CAS project {\sl Several Complex Variables and Complex Dynamics}. They both thank CAS for the support and for the wonderful atmosphere experienced there.

The authors also thank the referee for useful comments which improve the original manuscript.

\section{Geometry of unbounded convex domains}\label{geo-convex}

In this section we collect some simple results about geometry of unbounded convex domains.

If $x,y\in \C^n$, we denote by $[x,y]$ the real segment joining $x,y$, {\sl i.e.}, 
\[
[x,y]:=\{z\in \C^n: z=tx+(1-t)y, t\in [0,1]\}.
\]

For $z\in \C^n$ and $R>0$ we let 
\[
B(z,R):=\{w\in \C^n: \|w-z\|<R\}
\]
 be the Euclidean ball of center $z$ and radius $R$.

\begin{definition}
Let $D\subset \C^n$ be an unbounded convex domain. A vector $v\in \C^n$, $\|v\|=1$, is called a {\sl direction at $\infty$ for $D$} if there exists $x\in D$ such that $x+tv\in D$ for all $t\geq 0$.
\end{definition}

By convexity of $D$, if $v$ is a direction at $\infty$ for $D$, then for every $z\in D$ and all $t\geq 0$ it holds $z+tv\in D$  (see Lemma~2.1 in \cite{MS1}).

\begin{lemma}\label{two-components}
Let $D\subset \C^n$ be an unbounded convex domain. Then there exists at least one direction $v$ at $\infty$ for $D$.  Moreover 
\begin{enumerate}
\item either $D\setminus\overline{B(0,R)}$ has only one unbounded connected component for all $R>0$, 
\item or there exists $R_0>0$ such that $D\setminus\overline{B(0,R)}$ has two unbounded connected components for all $R\geq R_0$. This is the case if and only if the only directions at $\infty$ for $D$ are $v$ and $-v$. 
\end{enumerate}
\end{lemma}

\begin{proof}
Since $D$ is unbounded, there exists a sequence $\{p_k\}\subset D$ such that $\lim_{k\to \infty}\|p_k\|=\infty$. Up to subsequences, we can assume that $\lim_{k\to \infty}\frac{p_k}{\|p_k\|}=v$. Let $z\in D$. Since $D$ is convex, the real segment $[z,p_k]\subset D$ for all $k$.  Hence, since $\{z+tv,\ t \in \mathbb R^+\}$ is the limit, for the local Hausdorff convergence, of the segments $[z,p_k]$, then $\{z+tv,\ t \in \mathbb R^+\}$ is contained in $\overline{D}$. Finally, since $z \in D$, by convexity of $D$ it follows that $z+tv\in D$ for all $t\geq 0$.

Next, assume that there exists $R>0$ such that  $D\setminus\overline{B(0,R)}$ is not connected. We claim that $D\setminus\overline{B(0,R)}$ has at most two unbounded components. Indeed, if $U$ is an unbounded connected component of $D\setminus\overline{B(0,R)}$ and $\{p_k\}\subset U$ converges to $\infty$ and $\lim_{k\to \infty}\frac{p_k}{\|p_k\|}=v$ then for every $z\in U$ such that $\|z\|>R$ and for every $t\geq 0$ it holds $z+tv\in U$. Hence, for every unbounded connected component of $D\setminus\overline{B(0,R)}$ there exists $v\in \C^n$, $\|v\|=1$, such that $z+tv$ belongs to such a component for every $t\geq 0$ and some $z\in D$. If the unbounded components were more than two  there would exist two components $U$ and $U'$ and two directions $v$ and $w$  at $\infty$ for $D$ which are $\R$-linearly independent and such that $z_0+tv\in U$ for all $t\geq 0$ and some $z_0\in D$, and $z_1+tw\in U'$ for all $t\geq 0$ and some $z_1\in D$. But then, since $v$ and $w$ are $\R$-linearly independent, for $a, b$ sufficiently large the intersection $[z_0+av, z_1+bw]\cap \overline{B(0,R)}$ is empty and connects  $U$ with $U'$, a contradiction. 

Therefore, if $D\setminus\overline{B(0,R)}$ is not connected, then it has at most two unbounded connected components. If $D\setminus\overline{B(0,R)}$ contains two unbounded connected components, then it is easy to see that for every $R'>R$, also $D\setminus\overline{B(0,R')}$ contains two unbounded connected components. 

Moreover, the previous argument shows that if there are two $\R$-linearly independent directions at $\infty$, then for every $R>0$, $D\setminus\overline{B(0,R)}$ has only one unbounded connected component.  

Therefore, if 
$D\setminus\overline{B(0,R)}$ has two unbounded connected components, then there are  only two directions at $\infty$ for $D$, namely,  $v$ and $-v$ for some $v\in \C^n$, $\|v\|=1$. 

Conversely, assume $v, -v$ are the only directions at $\infty$ for $D$. Suppose by contradiction that  for every $R>0$ the open set $D\setminus\overline{B(0,R)}$ had only one unbounded connected component. Let $z_0\in D$, and $R>\|z_0\|$. Then there exists $t_R\in (0,\infty)$ such that   $z_0+tv,z_0-tv\in D\setminus\overline{B(0,R)}$ for all $t\geq t_R$.  Since $D\setminus\overline{B(0,R)}$ has only one unbounded connected component, the points  $z_0+t_Rv$ and $z_0-t_Rv$ can be joined by a continuous path $\gamma_R$ in $D\setminus\overline{B(0,R)}$. Let $H$ be the real affine hyperplane through $z_0$ orthogonal to $v$. Then, by construction, $H\cap \gamma_R\neq \emptyset$ for all $R\geq \|z_0\|$. Therefore, there exists a sequence $\{p_k\}\subset H\cap D$ converging to $\infty$ such that $w:=\lim_{k\to \infty}\frac{p_k}{\|p_k\|}$ is a direction at $\infty$ for $D$ which is $\R$-linearly independent of $v$, a contradiction.
\end{proof}

The previous lemma allows us to give the following definition:

\begin{definition}
Let $D\subset \C^n$ be an unbounded convex domain. We say that {\sl $D$ has one connected component at $\infty$} if for every $R>0$ the open set $D\setminus\overline{B(0,R)}$ has only one unbounded connected component. Otherwise we say that {\sl $D$ has two connected components at $\infty$}.
\end{definition}

\section{Hyperbolic geometry of the unit ball}\label{preli}

In this section we briefly recall what we need from hyperbolic geometry of the unit ball $\B^n$. We refer the reader to \cite{Abate, Kob} for details.

\subsection{Kobayashi distance and real Kobayashi geodesics} Given a domain $D\subset \C^n$, for every $z,w\in D$, we denote by $K_D(z,w)$ the {\sl Kobayashi distance} in $D$ between $z$ and $w$. If $D=\D\subset \C$ the unit disc, then $K_\D$ is the Poincar\'e distance of $\D$. If $D_1, D_2\subset \C^n$  are two domains and $f: D_1\to D_2$ is a biholomorphism, then $K_{D_1}(z,w)=K_{D_2}(f(z), f(w))$ for all $z,w\in D_1$. 

Let $D\subset \C^n$. A {\sl real (Kobayashi) geodesic} for $D$ is a piecewise $C^1$ curve $\gamma:(a,b)\to D$ such that for every $t,t'\in (a,b)$ it holds $K_D(\gamma(t), \gamma(t'))=|t-t'|$, where $-\infty\leq a<b\leq +\infty$. 

In the unit ball $\B^n$, given $z,w\in \B^n$, there exists a real geodesic $\gamma:[a,b]\to \B^n$ such that $\gamma(a)=z$ and $\gamma(b)=w$. Such a real geodesic is {\sl unique up to reparametrization}, in the sense that, if $\tilde\gamma:[\tilde a,\tilde b]\to \B^n$ is another real geodesic such that $\tilde \gamma(\tilde a)=z$ and $\tilde \gamma(\tilde b)=w$, then there exists  $c\in \R$ such that $\tilde \gamma(t)=\gamma(\pm t+c)$.  

The image of the real geodesic  joining $z,w\in \B^n$ can be described as follows. Let $\Delta:=\B^n\cap (\C(z-w)+w)$. Then $\Delta$ is an affine disc contained in the affine complex line $\C(z-w)+w$. Therefore, there exists an affine biholomorphic map $\varphi:\D \to \Delta$. Up to precomposing $\varphi$ with an automorphism of the unit disc, we can assume that $\varphi(0)=z$ and $\varphi(t)=w$ for some $t>0$. Then, $\gamma([a,b])=\varphi([0,t])$. In particular, it means that $\gamma([a,b])$ is either a real segment contained in a real line passing through the center of $\Delta$, or an arc of a circle in $\C(z-w)+w$ orthogonal to  the boundary of $\Delta$ within the set $\C(z-w)+w$. In other words, real geodesics in $\B^n$ are exactly the real geodesics for the Poincar\'e distance in $\Delta$.

In particular, if $D\subset\C^n$ is a domain biholomorphic to $\B^n$, then for every $z,w\in D$ there exists a unique (up to reparametrization) real (Kobayashi) geodesic joining $z$ and $w$. Moreover, if $F:\B^n\to D$ is a biholomorphism and $\gamma:[a,b]\to \B^n$ is a real geodesic in $\B^n$ joining $F^{-1}(a)$ and $F^{-1}(b)$, then $F\circ \gamma:[a,b]\to D$ is a real geodesic in $D$ joining $z$ and $w$.

By the previous consideration, it follows easily that, given $\xi\in \partial \B^n$ and $x_0\in \B^n$, there exists a unique real geodesic $\gamma_\xi:[0,+\infty)\to \B^n$ such that $\gamma_\xi(0)=x_0$ and $\lim_{t\to\infty}\gamma_\xi(t)=\xi$. Moreover, let $x_0\in \B^n$ and let $\{\zeta_k\}\subset \overline{\B^n}$ be a sequence converging to some point $\zeta\in \partial\B^n$. For each $k\in \N$, denote by $\gamma_k:[0,a_k)\to \B^n$ the unique real geodesic such that $\gamma_k(0)=x_0$ and either $\gamma_k(a_k)=\zeta_k$ in case $\zeta_k\in \B^n$ (and in such a case, necessarily $0<a_k<\infty$), or $\lim_{t\to \infty}\gamma_k(t)=\zeta_k$ in case $\zeta_k\in \partial\B^n$ (and in such a case, necessarily $a_k=\infty$). Finally, let $\gamma:[0,\infty)\to \B^n$ be the only real geodesic such that $\gamma(0)=x_0$ and $\lim_{t\to \infty}\gamma(t)=\zeta$. Then $\{\gamma_k\}$ converges uniformly on compacta of $[0,\infty]$ to $\gamma$. In other words, for every $\epsilon>0$ and $R>0$ there exists $k_0\in \N$ such that $a_k\geq R$ for every $k\geq k_0$ and, for every $s\in [0,R]$ and every $k\geq k_0$, it holds $K_{\B^n}(\gamma(s),\gamma_k(s))<\epsilon$.

Finally, if $\xi_1, \xi_2\in \partial \B^n$, $\xi_1\neq \xi_2$, there exists a unique (up to reparametrization) real geodesic $\gamma:(-\infty,+\infty)\to \B^n$ such that $\lim_{t\to -\infty}\gamma(t)=\xi_1$ and $\lim_{t\to \infty}\gamma(t)=\xi_2$. By the previous considerations, it is easy to see that, if  $\{\xi_k\}\subset \partial\B^n$ is a sequence converging to $\xi\in \partial \B^n$ and $\gamma_k:(-\infty,+\infty)$ is the unique (up to reparametrization) real geodesic whose closure contains $\xi_k$ and $\xi$, then for every  $\epsilon>0$ there exists $k_0\in \N$ such that $\gamma_k(-\infty,+\infty)\subset B(\xi, \epsilon)$ for all $k\geq k_0$.

\subsection{Horospheres} For $\xi\in \partial \B^n$ and $R>0$, let
\[
 E^{\B^n}(\xi,R):=\{z\in \B^n: \frac{|1-\langle z,\xi\rangle|^2}{1-\|z\|^2}<R\}.
\] 
The open set $E^{\B^n}(\xi,R)$ is called a {\sl horosphere of center $\xi$ and radius $R>0$}, and it is a complex ellipsoid affinely biholomorphic to $\B^n$. For the aim of this paper,  we need to recall the following properties of horospheres (see, {\sl e.g.}, \cite[Section 2]{Abate} or \cite{BG}):
\begin{enumerate}
\item $\overline{E^{\B^n}(\xi,R)}\cap \partial \B^n=\{\xi\}$ for every $R>0$, 
\item $\bigcup_{R>0} E^{\B^n}(\xi,R)=\B^n$,
\item $\bigcap_{R>0}\overline{E^{\B^n}(\xi,R)}=\{\xi\}$,
\item if $0<R<R'$ then $E^{\B^n}(\xi,R)\subset E^{\B^n}(\xi,R')$,
\item if $\xi_1, \xi_2\in \partial \B^n$ and $\xi_1\neq\xi_2$, then there exists $R>0$ such that $E^{\B^n}(\xi_1,R)\cap E^{\B^n}(\xi_2,R)=\emptyset$.
\end{enumerate}
If $\xi\in \partial \B^n$ and $\{\zeta_k\}\subset \B^n$ is any sequence converging to $\xi$, then for every $R>0$ it holds (see \cite[Prop. 2.2.20]{Abate}):
\[
E^{\B^n}(\xi,R)=\{z\in \B^n: \limsup_{k\to \infty}[K_{\B^n}(z,\zeta_k)-K_{\B^n}(0,\zeta_k)]<\frac{1}{2}\log R\}.
\]
If $F:\B^n\to D$ is a biholomorphism, and $\{z_k\}\subset D$ is any sequence such that $\{F^{-1}(z_k)\}$ converges to some $\xi\in\partial \B^n$, recalling that $F$ is an isometry with respect to the Kobayashi distance, the previous equation allows us to define the {\sl horosphere in $D$ of radius $R>0$, base point $F(0)$ relative to the sequence $\{z_k\}$} by setting
\begin{equation*}
\begin{split}
E^D(\{z_k\}, R)&:=F(E^{\B^n}(\xi, R))\\&=\{z\in D: \limsup_{k\to \infty}[K_{D}(z,z_k)-K_{D}(F(0),z_k)]<\frac{1}{2}\log R\}.
\end{split}
\end{equation*}
Horospheres defined using sequences can be used in general for defining a new topology (the {\sl horosphere topology}) and a new boundary for complete hyperbolic manifolds (see \cite{BG}). We content here to state the following result which is needed later on (see \cite[Proposition 6.1]{BG}):

\begin{lemma}\label{keepcpnv}
Let $F:\B^n \to D$ be a biholomorphism. Suppose $D\subset \C^n$ is a convex domain. Let $\{z_k\}\subset D$ be a sequence such that $\{F^{-1}(z_k)\}$ converges to some $\xi\in \partial \B^n$. Then for every $R>0$ the horosphere $E^D(\{z_k\}, R)$ is convex.
\end{lemma} 

\subsection{Iteration in the unit ball} Let $f:\B^n\to \B^n$ be a holomorphic map. If $f$ has no fixed points in $\B^n$ (namely, $f(z)\neq z$ for every $z\in \B^n$), then by  the Denjoy-Wolff theorem for $\B^n$ (due to M. Herv\'e \cite{He}, see also \cite{Mac} and  \cite[Theorem 2.2.31]{Abate}), there exists a unique point $\tau\in\partial\B^n$, which is called the {\sl Denjoy-Wolff point of $f$}, such that for every $z\in \B^n$ it holds $\lim_{k\to \infty}f^{\circ k}(z)=\tau$, where $f^{\circ k}:=f^{\circ (k-1)}\circ f$, $f^{\circ 1}=f$. In our argument we need the following result about Denjoy-Wolff points of commuting mappings (see \cite[Theorem~3.3]{Br}):

\begin{theorem}\label{filippo}
Let $f, g:\B^n\to \B^n$ be holomorphic and assume $f\circ g=g\circ f$. Let $\tau\in \partial \B^n$ be the Denjoy-Wolff point of $f$ and let $\sigma\in \partial \B^n$ be the Denjoy-Wolff point of $g$. Suppose $\tau\neq \sigma$. Let $\Delta:=\B^n\cap (\C(\sigma-\tau)+\tau)$. Then $f(\Delta)=g(\Delta)=\Delta$ and $f|_{\Delta}, g|_{\Delta}$ are hyperbolic automorphisms of $\Delta$ with fixed points $\sigma$ and $\tau$.
\end{theorem}

Note that $\Delta$ in the previous theorem is a disc contained in the affine complex line $\C(\sigma-\tau)+\tau$, hence there exists an affine biholomorphism $\varphi:\D\to \Delta$, and, saying that $f|_\Delta$ is a hyperbolic automorphism of $\Delta$, we mean that $\varphi^{-1}\circ f\circ \varphi:\D \to \D$ is a hyperbolic automorphism of $\D$. Recall also that a hyperbolic automorphism of $\D$ is an automorphism of $\D$ (hence a M\"obius transform) having exactly two fixed points on $\partial \D$ and no fixed points in $\D$. 

Finally, recall that a continuous one-parameter group $(h_t)$ of hyperbolic automorphisms of $\D$ is a continuous groups-homomorphism between the additive group of real numbers $\R$ endowed with the Euclidean topology and the group of automorphisms of $\D$ endowed with the topology of uniform convergence on compacta, such that, for every $t\neq 0$, the automorphism $h_t$ is hyperbolic (see, {\sl e.g.}, \cite{Abate} or \cite{Sh}). Then, there exist $\tau,\sigma\in \partial \D$, $\tau\neq \sigma$ such that $h_t(\tau)=\tau$  and $h_t(\sigma)=\sigma$ for all $t\in\R$. Also, $\tau$ is the Denjoy-Wolff point of $h_t$ for all $t>0$, while $\sigma$ is the Denjoy-Wolff point of  $h_t$ for all $t<0$ (or vice versa).

\begin{remark}\label{metocca}
Let $\gamma:(-\infty,+\infty)\to \D$ be a real (Poincar\'e) geodesic such that $\lim_{s\to -\infty}\gamma(s)=\tau$ and $\lim_{s\to \infty}\gamma(s)=\sigma$, and let $\Gamma:=\gamma(-\infty,+\infty)$. Since $h_t$ is an isometry for $K_\D$  and $\tau, \sigma$ are fixed points for $h_t$ for all $t\in \R$, then $h_t(\Gamma)=\Gamma$. Moreover, for $t>0$, given $s\in \R$ there exists $s'\in (-\infty,s)$ such that $h_t(\gamma(s))=\gamma(s')$ (that is, $h_t(\gamma(s))$ is closer to $\tau$ than $\gamma(s)$ along $\gamma$), while,  for $t<0$, given $s\in \R$ there exists $s'\in (s,\infty)$ such that $h_t(\gamma(s))=\gamma(s')$. Also, for every $\zeta_0,\zeta_1\in \Gamma$ there exists $t\in \R$ such that $h_t(\zeta_0)=\zeta_1$.
\end{remark}

\section{Extension of convex maps}

In this section we prove that every convex map of the unit ball extends continuously up to the boundary  from the closed unit ball to the  one-point compactification of $\C^n$. To this aim, we need some preliminary lemmas.

The following lemma was proved in \cite[Lemma 6.16]{BG}:
\begin{lemma}\label{zimmer}
Let $F:\B^n\to D$ be a biholomorphism. Suppose $D$ is convex. Let $p, q\in \partial D$, $p\neq q$. Then for every sequences $\{p_n\}, \{q_n\}\subset D$ such that $\lim_{n\to \infty} p_n=p$ and $\lim_{n\to \infty} q_n=q$ it holds
\[
\lim_{n\to \infty} K_D(p_n,q_n)=\infty.
\] 
\end{lemma}
 
If $V\subset D$ and $\epsilon>0$, we let
\[
\mathcal N_\epsilon(V):=\{z\in D: \exists w\in V, K_D(z,w)<\epsilon\}.
\]

\begin{lemma}\label{herve}
Let $F:\B^n\to D$ be a biholomorphism. Suppose $D$ is convex. Let $x\in D$ and let $p\in \partial D$. Then there exist an open set $U$ containing $p$ and $M>0$  such that  for every sequence $\{p_k\}\subset D\cap U$ converging to $p$, the real (Kobayashi) geodesic $\gamma_k:[0,R_k]\to D$ such that 
$\gamma_k(0)=x$ and $\gamma_k(R_k)=p_k$ satisfies
\[
\gamma_k(s)\in \mathcal N_M([x, p_k]) \quad \forall s\in [0, R_k].
\]
\end{lemma}
\begin{proof}
By \cite[Lemma 6.17]{BG}, there exist $A>0$ and $B>0$ such that for every $k\in \N$, the real segment $[p_k, x]$ is a $(A,B)$-quasi-geodesic in the sense of Gromov (see, {\sl e.g.}, \cite[Section 6.2]{BG}, \cite{Zim}, \cite{GH}, \cite{Gr}). Therefore the statement of the lemma follows immediately from  Gromov's shadowing lemma  (see \cite[Th\'eor\`eme 11 p. 86]{GH}), since $(D,K_D)$ is Gromov hyperbolic because so is $(\B^n, K_{\B^n})$  and $F$ is an isometry for the Kobayashi distance.
\end{proof}

\begin{definition}
For a domain $D\subset \C^n$ we denote by $\overline{D}^\ast$ its closure in the one point compactification of $\C^n$. 
\end{definition}
Clearly, if $D$ is relatively compact in $\C^n$ then $\overline{D}^\ast=\overline{D}$, while, if $D$ is unbounded, then $\overline{D}^\ast=\overline{D}\cup \{\infty\}$. 

\begin{proposition}\label{continuous-ext}
Let $F:\B^n \to D$ be a biholomorphism. Suppose $D$ is convex.  Then there exists a continuous map $\tilde{F}:\overline{\B^n}\to \overline{D}^\ast$ such that $\tilde{F}|_{\B^n}=F$.
\end{proposition}

\begin{proof} We show, and it is enough, that for every $\xi\in \partial \B^n$ either the limit $\lim_{z\to \xi}F(z)$ exists and belongs to $\partial D$, or $\lim_{z\to \xi}\|F(z)\|=\infty$.

Assume by contradiction that this is not true. Then there exists $\xi\in \partial \B^n$ and two sequences $\{z^1_k\}, \{z^2_k\}\subset \B^n$ converging to $\xi$ such that either $\lim_{k\to \infty}F(z^j_k)= p_j\in \partial D$, $j=1,2$,  for some $p_1\neq p_2$, or $\lim_{k\to \infty}F(z^1_k)= p_1\in \partial D$ and $\lim_{k\to \infty}\|F(z^2_k)\|= \infty$. In the second case, since the cluster set of $F$ at $\xi$ is connected, we can find another point $p_2\in \partial D$, $p_2\neq p_1$ and another sequence $\{w_k\}\subset \B^n$ converging to $\xi$ such that $\lim_{k\to \infty} F(w_k)=p_2$. Therefore, we only need to consider the first case. Let $p^j_k:=F(z^j_k)$, $k\in \N$ and $j=1,2$.

Let $x_0\in D$. For every $k\in \N$, let $\gamma_k^j:[0,R^j_k]\to D$  be the unique real (Kobayashi) geodesic such that $\gamma_k^j(0)=x_0$ and $\gamma_k^j(R^j_k)=p^j_k$, $k\in \N$ and $j=1,2$.  For $j=1,2$, let 
\[
V_j:=\bigcup_{k\in \N} [x_0, p^j_k].
\]
Up to replace $\{z_k^j\}_{k\in \N}$ with $\{z_k^j\}_{k\geq k_0}$ for some large $k_0\in \N$ in such a way that $\{p_k^j\}$ is all contained in a small neighborhood of $p_j$, $j=1,2$, by  Lemma \ref{herve} there exists $M>0$ such that for every $k\in \N$, for every $s\in [0, R^j_k]$ and $j=1,2$, it holds
\begin{equation}\label{Eq1}
\gamma_k^j(s)\in \mathcal N_M(V_j).
\end{equation}
Since $F$ is an isometry for the Kobayashi distance, it follows that $F^{-1}\circ \gamma_k^j:[0, R_k^j]\to \B^n$ is the unique real geodesic joining $F^{-1}(x_0)$ and $z_k^j$, $j=1,2$. Since $\{z_k^j\}$ converges to $\xi$ as $k\to \infty$, $j=1,2$, it follows that  $F^{-1}\circ \gamma_k^j$ converges uniformly on compacta of $[0,\infty)$ to the unique real geodesic $\tilde\gamma$ in $\B^n$ joining $F^{-1}(x_0)$ with $\xi$. Therefore, if we let $\gamma:=F \circ \tilde\gamma$, since $F$ is an isometry for the Kobayashi distance, it follows that for every $R>0$ there exists $k_R\in \N$ such that for every $s\in [0,R]$ and for every $k\geq k_R$ it holds for $j=1,2$, 
\begin{equation}\label{eq2}
K_{D}(\gamma^j_k(s), \gamma(s))<M.
\end{equation}
By the triangle inequality, \eqref{Eq1} and \eqref{eq2} imply that
\[
\gamma(s) \in \mathcal N_M(V_1)\cap \mathcal N_M(V_2),\quad \forall s\in [0,\infty).
\] 
In particular,  this implies that  there exist two sequences $\{q^j_k\}\subset V_j$, $j=1,2$, such that for every $k\in \N$ and $j=1,2$ it holds
\begin{equation}\label{Eq3}
K_D(\gamma(k), q^j_k)<M.
\end{equation}
Since the sequence $\{\gamma(k)\}_{k\in \N}$ is not relatively compact in $D$ and $D$ is complete hyperbolic, it follows that $\{q_k^j\}$ is not relatively compact in $D$, $j=1,2$. Using the convexity of $D$ it is not hard to see that the only possibility is that $\lim_{k\to \infty} q_k^j=p_j$, $j=1,2$. But, by the triangle inequality, \eqref{Eq3} implies that 
\[
\lim_{k\to \infty} K_D(q_k^1, q_k^2)<2M,
\]
contradicting Lemma \ref{zimmer}. 
\end{proof}

\section{Infinite discontinuities}

\begin{definition}
Let $F:\B^n \to \C^n$ be holomorphic. We say that $\xi\in \partial \B^n$ is an {\sl infinite discontinuity} of $F$ if $\lim_{z\to \xi} \|F(z)\|=\infty$.
\end{definition}

\begin{proposition}\label{singular}
Let $F:\B^n\to D$ be a biholomorphism. Suppose $D$ is convex and unbounded. Then $F$ has either one or two infinite discontinuities. More precisely,
\begin{enumerate}
\item if $D$ has one connected component at $\infty$ then $F$ has  one infinite discontinuity,
\item if $D$ has two connected components at $\infty$  then  $F$ has two infinite discontinuities. 
\end{enumerate}
\end{proposition}

\begin{proof}
By Lemma \ref{two-components}, there exists $v\in \C^n$ a direction at $\infty$ for $D$. Then for every $z\in D$ and $t\geq 0$, it holds $z+tv\in D$. Therefore, for every $t\geq 0$, the map $\phi_t:\B^n\to \B^n$, given by
\[
\phi_t(\zeta):=F^{-1}(F(\zeta)+tv),
\]
is well defined and univalent. Moreover, it is easy to see that $\phi_0={\sf id}_{\B^n}$, that $\phi_{t+s}=\phi_t\circ \phi_s$ for every $t, s\geq 0$ and that $[0,\infty)\ni t\mapsto \phi_t$ is continuous with respect to the Euclidean topology in $[0,\infty)$ and the topology of uniform convergence on compacta of ${\sf Hol}(\B^n, \B^n)$. In other words, $(\phi_t)_{t\geq 0}$ is a continuous one-parameter semigroup of $\B^n$. Moreover, by definition,  $\phi_t$ has no fixed points in $\B^n$ for $t>0$. By the continuous version of the Denjoy-Wolff theorem for continuous one-parameter semigroups of $\B^n$ (see \cite{Abate, AS, BCD}), there exists a unique point $\xi\in \partial \B^n$ (which is in fact the Denjoy-Wolff point of $\phi_t$ for every $t>0$) such that for every $\zeta\in \B^n$ it holds $\lim_{t\to \infty}\phi_t(\zeta)=\xi$. Unrolling the definition of $\phi_t$, this means that for every $z\in D$ we have
\[
\lim_{t\to \infty}F^{-1}(z+tv)=\xi.
\]
Now, let us assume that there exists another direction at $\infty$ for $D$, say $w$, which is $\R$-linearly independent of $v$. Then we can define another continuous one-parameter semigroup of $\B^n$, call it $(\psi_s)_{s\geq 0}$, by setting
\[
\psi_s(\zeta):=F^{-1}(F(\zeta)+sw).
\]
Let $\xi':=\lim_{s\to \infty}F^{-1}(z+sw)$ for some--and hence, as we just proved, for every--$z\in D$. We claim that $\xi=\xi'$. 

Suppose by contradiction that $\xi\neq\xi'$. It is easy to see that for every $s, t\geq 0$ the maps $\phi_t$ and $\psi_s$ commute, {\sl i.e.}, $\phi_t\circ \psi_s=\psi_s\circ \phi_t$. Let $\Delta:=(\C(\xi-\xi')+\xi)\cap \B^n$. Since the Denjoy-Wolff point of $\phi_t$  is $\xi$ for all $t> 0$ and the Denjoy-Wolff point of $\psi_s$ is $\xi'$ for all $s>0$, and we are supposing $\xi\neq \xi'$, by Theorem \ref{filippo}, it follows that for all $s, t\geq 0$,
\[
\phi_t(\Delta)=\Delta, \quad \psi_s(\Delta)=\Delta
\]
and the restriction of $\phi_t$ and $\psi_s$ to $\Delta$ are hyperbolic automorphisms of $\Delta$. Let $\gamma:(-\infty, +\infty)\to \Delta$ be a parameterization of the real  geodesic in $\Delta$ whose closure contains $\xi$ and $\xi'$. Then, it follows that $\phi_t\circ \gamma:(-\infty,\infty)\to \Delta$ and $\psi_s\circ \gamma:(-\infty,\infty)\to \Delta$ are other parameterizations for the same real Kobayashi geodesic whose closure contains $\xi$ and $\xi'$, for all $s, t\geq 0$. Moreover, since $\xi$ is the Denjoy-Wolff point of $\phi_t|_\Delta$ and $\xi'$ is the Denjoy-Wolff point of $\psi_s|_\Delta$, it follows that (see Remark \ref{metocca}), letting $\zeta_0:=\gamma(0)$,  there exists $s_0\in (0,\infty)$ such that $\psi_{s_0}(\phi_1(\zeta_0))=\zeta_0$. Thus, 
\[
F(\zeta_0)=F(\psi_{s_0}(\phi_1(\zeta_0)))=F(\zeta_0)+s_0 w+v,
\]
that is, $s_0 w+v=0$, which is a contradiction since $w$ and $v$ are assumed to be $\R$-linearly independent.

Summing up we proved that for every two directions $v, w$ at $\infty$ for $D$ which are $\R$-linearly independent it holds
\begin{equation}\label{v-w-limit}
\lim_{t\to \infty}F^{-1}(z_0+tw)=\lim_{t\to \infty}F^{-1}(z_1+tv), \quad \forall z_0, z_1\in D.
\end{equation}

Now, let $\{p_k\}\subset D$ be a sequence converging to $\infty$, and assume that $\lim_{k\to \infty} \frac{p_k}{\|p_k\|}=w$, for some $w\in \C^n$ direction at $\infty$ for $D$. Let $z_0\in D$. We claim that
\begin{equation}\label{sequence-w}
\lim_{k\to \infty}F^{-1}(p_k)=\lim_{t\to \infty}F^{-1}(z_0+tw).
\end{equation}
Assume this is not the case and, up to subsequences, $\lim_{k\to \infty}F^{-1}(p_k)=\xi_0\in \partial \B^n$ and  $\lim_{t\to \infty}F^{-1}(z_0+tw)=\xi_1\in \partial \B^n$ with $\xi_0\neq \xi_1$. 

Let fix $R>0$ and let $E:=E^{\B^n}(\xi_0, R)$ be the horosphere of center $\xi_0$ and radius $R$ in $\B^n$. Then $F(E)=E^D(\{p_k\}, R)$  is convex by Lemma \ref{keepcpnv}. Also, there exists $t_0\geq 0$ such that $z_0+tw\not\in F(E)$ for all $t\geq t_0$, because $F^{-1}(z_0+tw)$ does not converge to $\xi_0$ as $t\to \infty$ and hence it is eventually outside $E$. Moreover, $F(E)$ is unbounded because by Proposition \ref{continuous-ext}, $\lim_{z\to \xi_0}\|F(z)\|=\lim_{k\to \infty} \|F(F^{-1}(p_k))\|=\infty$. By Lemma \ref{two-components}, there exists a direction at $\infty$ for $F(E)$, call it $u$, that is, $z+tu\in F(E)$ for every $z\in F(E)$ and for every $t\geq 0$. Since $\overline{E}\cap \partial \B^n=\{\xi_0\}$, it follows that $\lim_{t\to \infty}F^{-1}(z+tu)=\xi_0$. 

Hence, there are two cases: either $u$ is $\R$-linearly independent of $w$, or $u=-w$. In the first case, we contradict \eqref{v-w-limit}. Therefore, $u=-w$. In this case, note that for every $R>0$ there exist $k_R\in \N$ and $t_R\geq 0$ such that $p_k$ and $z_0+tw$ are  in the same unbounded connected component of $D\setminus\overline{B(0,R)}$ for every $k>k_R$, and $t>t_R$, call $U_R$ such an unbounded component. Hence, $\bigcap_{R>0} \overline{F^{-1}(U_R)}$ is compact and connected in $\partial \B^n$ and contains $\xi_0$ and $\xi_1$. Since $\xi_0\neq \xi_1$, this implies that there exists a sequence $\{q_k\}\subset D$ converging to $\infty$ which is eventually contained in $U_R$ for all $R>0$ such that $\{F^{-1}(q_k)\}$ converges to a point $\xi_2\in \partial \B^n$ with $\xi_2\neq \xi_0$ and $\xi_2\neq \xi_1$. We can then repeat the previous argument considering $\xi_2$ instead of $\xi_1$ and $\{q_k\}$ instead of $\{p_k\}$. But this time,  the direction $u$ at $\infty$ for $F(E)$ can not be $w$ nor $-w$, since, by construction, both $F^{-1}(z_0+tw)$ and $F^{-1}(z_0-tw)$ are not eventually contained in $E=E^{\B^n}(\xi_2,R)$ for $t\geq 0$, hence, we are back to the first case and again we  contradict \eqref{v-w-limit}. Therefore, \eqref{sequence-w} holds. 

Suppose now that $D$ has one connected component at $\infty$. If for every direction $v$ at $\infty$ for $D$, the vector $-v$ is not a direction at $\infty$ for $D$, then it follows immediately from \eqref{sequence-w} and \eqref{v-w-limit} that $F$ has only one infinite singularity. In case $v$ and $-v$ are directions at $\infty$ for $D$, we claim that  for every $z_0 \in D$:
\[
\lim_{t\to \infty}F^{-1}(z_0+tv)=\lim_{t\to \infty}F^{-1}(z_0-tv).
\]
Indeed, by Lemma \ref{two-components}, since $D$ has one connected component at $\infty$, there exists some other direction $w$ at $\infty$ for $D$ such that $w\neq \pm v$. Therefore, $w$  is  $\R$-linearly independent of $v$ (and then of $-v$). Hence by \eqref{v-w-limit}
\[
\lim_{t\to \infty}F^{-1}(z_0+tv)=\lim_{t\to \infty}F^{-1}(z_0+tw)=\lim_{t\to \infty}F^{-1}(z_0-tv).
\]
From this and from \eqref{sequence-w} and \eqref{v-w-limit} it follows again that $F$ has a unique infinite discontinuity.

On the other hand, if $D$ has two connected components at $\infty$, the only directions at $\infty$ for $D$ are $v$ and $-v$. We claim in this case that  for every $z_0 \in D$:
\begin{equation}\label{vo-inf}
\lim_{t\to \infty}F^{-1}(z_0+tv)\neq\lim_{t\to \infty}F^{-1}(z_0-tv).
\end{equation}
Once we proved the claim (\ref{vo-inf}), \eqref{sequence-w} implies at once that $F$ has exactly two infinite discontinuities. 

In order to prove the claim, assume by contradiction that $\xi:=\lim_{t\to \infty}F^{-1}(z_0+tv)=\lim_{t\to \infty}F^{-1}(z_0-tv)$. 
By Lemma \ref{two-components}, there exists $R_0$ such that for every $R\geq R_0$ the open set $D\setminus \overline{B(0,R)}$ has two unbounded connected components, say $U_1, U_2$. Up to relabelling, it is clear that $z_0+tv$ is eventually contained in $U_1$ and  $z_0-tv$ is eventually contained in $U_2$, for $t$ large. Let $K:=\overline{F^{-1}(D\cap \overline{B(0,R)})}$. Since $\lim_{z\to \xi}\|F(z)\|=\infty$ by Proposition~\ref{continuous-ext}, it follows that $\xi\not\in K$, while $\zeta^+_t:=F^{-1}(z_0+tv)$ and $\zeta^-_t:=F^{-1}(z_0-tv)$ are close to $\xi$ for $t$ large. Therefore, the real (Kobayashi) geodesic $\gamma_t$ joining $\zeta^+_t$  and $\zeta^-_t$ satisfies $\gamma_t\cap K=\emptyset$ for $t$ sufficiently large. But then $F(\gamma_t)$ is a continuous path in $D\setminus \overline{B(0,R)}$ which joins $z_0+tv$ and $z_0-tv$, against the fact that $z_0+tv\in U_1$ and $z_0-tv\in U_2$. Thus, \eqref{vo-inf} holds.
\end{proof}

\section{Homeomorphic extension of convex maps}\label{final}

In this section we collect the previous results and prove our main theorem, from which Theorem \ref{main-intro} follows at once:

\begin{theorem}\label{main}
Let $F:\B^n \to D$ be a biholomorphism. Suppose that $D$ is convex. Then 
\begin{enumerate}
\item $D$ is bounded and $F$ extends as a homeomorphism $\tilde{F}: \overline{\B^n}\to \overline{D}$, or
\item $D$ is unbounded and has one connected component at $\infty$, $F$ extends as a homeomorphism $\tilde{F}: \overline{\B^n}\to \overline{D}^\ast$. In particular, $F$ has only one infinite discontinuity $\xi\in \partial \B^n$ such that $\lim_{z\to \xi}\|F(z)\|=\infty$ and $F$ extends as a homeomorphism $\tilde{F}: \overline{\B^n}\setminus \{\xi\}\to \overline{D}$. Or,
\item  $D$ is unbounded and has two connected components at $\infty$, $F$ has two infinite discontinuities $\xi_1, \xi_2\in \partial \B^n$, $\xi_1\neq\xi_2$, such that $\lim_{z\to \xi_1}\|F(z)\|=\lim_{z\to \xi_2}\|F(z)\|=\infty$ and $F$ extends as a homeomorphism $\tilde{F}: \overline{\B^n}\setminus \{\xi_1, \xi_2\}\to \overline{D}$.
\end{enumerate}
\end{theorem}
\begin{proof}
By Proposition \ref{continuous-ext}, $F$ has a continuous extension $\tilde{F}: \overline{\B^n}\to \overline{D}^\ast$. If $D$ is unbounded, by Proposition~\ref{singular}, $\tilde{F}^{-1}(\infty)$ contains either one point (in case $D$ has one connected component at $\infty$) or two points (in case $D$ has two connected components at $\infty$). 

We show that, for every $x_1, x_2\in \partial \B^n$, $x_1\neq x_2$, such that $\tilde{F}(x_j)\neq \infty$, $j=1,2$, it follows $\tilde{F}(x_1)\neq \tilde{F}(x_2)$. The proof of this is similar to that of \cite[Corollary 8.3]{BG}, but we sketch it here for the reader convenience.

Assume by contradiction that there exist $x_1, x_2\in \partial \B^n$ such that $p=\tilde F(x_1)=\tilde F(x_2)\in \partial D$.  Let $R>0$ be such that $V:=E^{\B^n}(x_1, R)\cap E^{\B^n}(x_2, R)\neq\emptyset$. Then $V$ is an open, convex, relatively compact  subset of $\B^n$. Since $E_j:=F(E^{\B^n}(x_j, R))$, $j=1,2$, is a horosphere in $D$, hence it is convex by Lemma \ref{keepcpnv}, it follows that $F(V)$ is an open, convex, relatively compact  subset of $D$.  

Now, let $\{\zeta_k^j\}_{k\in \N}\subset E^{\B^n}(x_j, R)$,  be a sequence converging to $x_j$, $j=1,2$. Hence by hypothesis, $p=\lim_{k\to \infty}F(\zeta_k^j)$, $j=1,2$, showing that $p\in \overline{E_1}\cap \overline{E_2}$. Therefore, given any $z_0\in V$, it follows that $[z_0,p)\subset E_j$, $j=1,2$. Hence, $[z_0,p)\subset F(V)$, which is not relatively compact in $D$, a contradiction.

Thus, if either $D$ is relatively compact in $\C^n$, or if $D$ has one connected component at $\infty$, it follows that $\tilde{F}:\overline{\B^n}\to \overline{D}^\ast$ is a bijective continuous map from a compact space to a Hausdorff space, hence, a homeomorphism. 

In case $D$ has two connected components at $\infty$, let $\xi_1, \xi_2\in \partial \B^n$ be the two infinite discontinuities of $F$. Since every closed subset $C\subset \overline{\B^n}\setminus\{\xi_1, \xi_2\}$ is also compact in $\overline{\B^n}\setminus\{\xi_1, \xi_2\}$, it follows that $\tilde{F}(C)$ is compact in $\overline{D}$, and hence closed. Therefore, $\tilde{F}: \overline{\B^n}\setminus\{\xi_1, \xi_2\}\to \overline{D}$ is a closed, bijective, continuous map, thus it is a homeomorphism. 
\end{proof}


\begin{thebibliography}{99}
\bibitem{Abate} M. Abate, {\sl Iteration theory of holomorphic maps on taut manifolds}, Mediterranean Press, Rende, 1989.
\bibitem{AS} L. Aizenberg, D. Shoikhet, {\sl Boundary behavior of semigroups of holomorphic mappings on the
unit ball in $\C^n$}, Complex Variables Theory Appl. 47 (2002), 109-121.
\bibitem{Br} F. Bracci, {\sl Common fixed points of commuting holomorphic maps in the unit ball of $\C^n$}, Proc. Amer. Math. Soc. 127 (1999), 1133-1141. 
\bibitem{BCD} F. Bracci, M. D. Contreras, S. D\'iaz-Madrigal, {\sl Classification of semigroups of linear fractional maps in the unit ball}.  Adv. Math., 208, (2007), 318-350.
\bibitem{BG} F. Bracci, H. Gaussier, {\sl Horosphere topology}, ArXiv.1605.04119v4
\bibitem{GH} E. Ghys, P. de La Harpe, {\sl Sur les groupes hyperboliques d'apr\`es Mikhael Gromov}. Progress in Mathematics {\bf 83}, Birkhauser.
\bibitem{Gr} M. Gromov, {\sl Hyperbolic groups. Essays in group theory.} Math. Sci. Res. Inst. Publ., 8, Springer, New York (1987), 75-263.
\bibitem{He} M. Herv\'e, {\sl Quelques propri\'et\'es des applications analytiques d'une boule \`a m dimensions dans
elle-m\^eme}. J. Math. Pures Appl. 42 (1963), 117-147.
\bibitem{Kob} S. Kobayashi, {\sl Hyperbolic Complex Spaces}, Springer-Verlag, Grundlehren der mathematischen Wissenschaften 318, 1998.
\bibitem{Lem} L. Lempert, {\sl La m\'etrique de Kobayashi et la r\'epr\'esentation des domaines sur la boule}. Bull. Soc. Math. Fr. 109 (1981), 427-474.
\bibitem{Mac} B. D. MacCluer, {\sl Iterates of holomorphic self-maps of the unit ball in $\C^N$}. Mich. Math. J. 30 (1983), 97-106.
\bibitem{MS1} J. R. Muir, Jr., T. J. Suffridge, {\sl Unbounded convex mappings of the ball in $\C^n$}, Proc. Amer. Math. Soc., 129 (2001), no. 11, pp. 3389-3393.
\bibitem{MS2}  J. R. Muir, Jr., T. J. Suffridge, {\sl A generalization of half-plane mappings to the ball in $\C^
N$}. Trans. Amer. Math. Soc. 359, (2007), 1485-1498.
\bibitem{Sh} D. Shoikhet, {\sl Semigroups in Geometrical Function Theory}, Kluwer, Dordrecht, NL, 2001.
\bibitem{Zim} A. M. Zimmer, {\sl Gromov hyperbolicity and the Kobayashi metric on convex domains of finite type}. Math. Ann. 365 (2016), 1425-1498.
\end{thebibliography}
\end{document}